\documentclass[notitlepage]{article}
\usepackage{amsopn} 
\usepackage{amsfonts} 
\usepackage{amsmath} 
\usepackage{amssymb} 
\usepackage{amscd}
\usepackage{amsthm}
\usepackage{enumerate} 
\usepackage{verbatim}
\usepackage{url}
\usepackage{xspace}
\newtheorem{lemma}{Lemma}
\newtheorem{example}{Example}
\newtheorem{theorem}{Theorem}

\renewcommand{\epsilon}{\varepsilon}

\newcommand{\calD}{\ensuremath{\mathcal D}\xspace}

\newcommand{\calM}{\ensuremath{\mathcal{M}}\xspace}

\newcommand{\calR}{\ensuremath{\mathcal{R}}\xspace}
\newcommand{\calS}{\ensuremath{\mathcal{S}}\xspace}

\newcommand{\N}{\ensuremath{\mathbb{N}}\xspace}

\newcommand{\Z}{\ensuremath{\mathbb{Z}}\xspace}

\newcommand{\bfc}{\ensuremath{\mathbf{c}}\xspace}
\newcommand{\bfd}{\ensuremath{\mathbf{d}}\xspace}

\newcommand{\bfl}{\ensuremath{\mathbf{l}}\xspace}

\newcommand{\bfq}{\ensuremath{\mathbf{q}}\xspace}

\newcommand{\bfs}{\ensuremath{\mathbf{s}}\xspace}
\newcommand{\bft}{\ensuremath{\mathbf{t}}\xspace}
\newcommand{\bfu}{\ensuremath{\mathbf{u}}\xspace}

\newcommand{\bfw}{\ensuremath{\mathbf{w}}\xspace}
\newcommand{\bfx}{\ensuremath{\mathbf{x}}\xspace}
\newcommand{\bfy}{\ensuremath{\mathbf{y}}\xspace}
\newcommand{\bfz}{\ensuremath{\mathbf{z}}\xspace}

\newcommand{\nsset}[1]{\{#1\}} 

\newcommand{\nset}[2]{\{\,#1\mid#2\,\}} 
\newcommand{\bset}[2]{\bigl\{\,#1\mid #2 \,\bigr\}}
 
\newcommand{\nparen}[1]{(#1)}
\newcommand{\bparen}[1]{\bigl(#1\bigr)}

\newcommand{\nabs}[1]{\lvert #1\rvert}

\newcommand{\qtext}[1]{\quad\text{#1}\quad}
\newcommand{\qqtext}[1]{\qquad\text{#1}\qquad}

\newcommand{\ktoinf}{k \rightarrow \infty}

\newcommand{\ntoinf}{n \rightarrow \infty}

\newcommand{\word}[1]{\ensuremath{\mathtt{#1}}\xspace}

\usepackage{hyperref}
\hypersetup{%
pdftitle={},%
bookmarks, bookmarksopen,%
colorlinks, linkcolor=blue,%
hyperindex=true%
}
\renewcommand{\o}{\ensuremath{\mathtt{0}}\xspace}

\renewcommand{\i}{\ensuremath{\mathtt{1}}\xspace}
\newcommand{\ka}{\ensuremath{\mathtt{2}}\xspace}
\newcommand{\ko}{\ensuremath{\mathtt{3}}\xspace}

\newcommand{\bina}{\nsset{\o,\i}}

\begin{document}	
\title{Extremal words in morphic subshifts}

\author{James D. Currie\thanks{j.currie@uwinnipeg.ca}
\and
Narad Rampersad\thanks{narad.rampersad@gmail.com}
\and
Kalle Saari\thanks{kasaar2@gmail.com}
\\
Department of Mathematics and Statistics\\
University of Winnipeg\\
515 Portage Avenue\\
Winnipeg, MB, R3B 2E9, Canada
\\\\
Luca Q. Zamboni\thanks{lupastis@gmail.com}\\
FUNDIM\\
University of Turku, Turku\\
FIN-20014, Finland\\
and\\
Institut Camille Jordan\\
Universit\'e Lyon 1, France}

\date{}

\maketitle

\begin{abstract}
Given an infinite word $\bfx$ over an alphabet $A$, a letter $b$ occurring in $\bfx$,  and a total order $\sigma$ on $A$, we call the
smallest word with respect to $\sigma$ starting with $b$ in the shift orbit closure of $\bfx$ 
an \emph{extremal word} of~$\bfx$. In this paper we consider the extremal words of morphic words.
If $\bfx = g(f^{\omega}(a))$ for some morphisms $f$ and $g$, we give two simple conditions on $f$ and $g$ that guarantees that all extremal words are morphic.
This happens, in particular, when $\bfx$ is a primitive morphic or a  binary pure morphic word. 
Our techniques provide characterizations of the extremal words of the Period-doubling word and the Chacon word
and give a new proof of the form of the lexicographically least word in the shift orbit closure of the Rudin-Shapiro word.
\end{abstract}
\textbf{keywords:} 
Lexicographic order, morphic word, primitive morphic word, extremal word, Period-doubling word, Chacon word, Rudin-Shapiro word

\section{Introduction}

Given an infinite word $\bfx \in A^{\N}$, it is natural to inquire about the nature of the lexicographically least words in its shift orbit closure.
We call a word $\bfy$ in the shift orbit closure of $\bfx$ \emph{extremal} if there exists a total order $\sigma$ on $A$ and a letter $b\in A$ such that
$\bfy$ is the least word with respect to $\sigma$ beginning in $b$.

We get different extremal words depending on the choice of  $\sigma$ and $b$.
For example, if $A = \{\o, \i\}$ and $\bfx$ is a Sturmian word, it is well-known that the extremal words with respect to $\o < \i$ 
are $\o\bfc$ and $\i\o\bfc$, where $\bfc$ is the characteristic word whose slope equals that of $\bfx$, and if we order $\i < \o$, then the extremal words are 
$\i \bfc$ and $\o\i\bfc$; see, for example,~\cite{Pirillo2005}. As another example, if $\bfx$ is $k$-automatic for some positive integer $k$, then its extremal words are $k$-automatic
as well~\cite{AllRamSha2009}. We refer the reader to~\cite{Allouche1983,AllCurSha1998,AllCos1983,Gan2000} for related results.

These considerations motivate the following question: given a morphic word, is it true that the corresponding extremal words are morphic, too?
While we are not able to solve the question in full generality, we give two fairly general 
classes of morphic words in which the answer is affirmative:
In Theorem~\ref{su 04-11-2012 12:52} we show that if $\bfx\in A^{\N}$  is a pure morphic word generated by
a morphism in a certain set $\calM_{\bfx}$ defined in~\eqref{ti 15-01-2013 09:55},  then all extremal words of $\bfx$ are morphic.
In Theorem~\ref{2013-07-13 19:29} we show that if $\bfx$ is a primitive morphic word, that is, a morphic image of a fixed point of a primitive morphism,
then the extremal words are primitive morphic as well.
We will also show that the extremal words of all binary pure morphic words are morphic (Theorem~\ref{ma 14-01-2013 10:22}). 
Finally, we give characterizations of the extremal words of the Period-doubling word (Theorem~\ref{14-01-2013 15:05}) and Chacon word (Theorem~\ref{14-01-2013 15:07}).
Along the way we show that if $\bfx$ is a pure morphic word generated by a morphism $f$ and $\bft$ is in the shift orbit closure of $\bfx$ such that $f(\bft) = \bft$, then $\bft$ is morphic (Theorem~\ref{su 13-01-2013 16:00}).

\section{Preliminaries} 

We will follow the standard terminology and notation of combinatorics on words as established, for example, in~\cite{AllSha2003,Lothaire2002}.

Let $A$ be a finite alphabet. We write $A^{\N}$  for the set of all infinite words over $A$. If $X \subset A^{*}$, then $X^{\omega}$ denotes the set of all infinite words obtained by a concatenation of words in $X$.

If $f \colon A^{*} \rightarrow A^{*}$ is a morphism such that $f(a) = ax$ for some letter $a\in A$ and a word $x\in A^{+}$ such that $f^{n}(x) \neq \epsilon$, the empty word, for all $n\geq 0$, then there exists an infinite word $f^{\omega}(a) := \lim_{\ntoinf} f^{n}(a)$ such that $f^{n}(a)$ is a prefix of $f^{\omega}(a)$ for all $n\geq 0$, and it is called a \emph{pure morphic word generated by} $f$. If $c \colon A^{*} \rightarrow B^{*}$ is a coding, that is  a letter-to-letter morphism, then $c(f^{\omega}(a))$ is called a \emph{morphic word}.
Notice that $f^{\omega}(a)$ is a fixed point of $f$, that is $f\bparen{f^{\omega}(a)} = f^{\omega}(a)$, but in general a fixed point of a morphism is not necessarily generated by the morphism (however, see Theorem~\ref{su 13-01-2013 16:00}).

A  morphism $f$ is called \emph{primitive} if there exists an integer $k\geq1$ such that $b$ occurs in $f^k(a)$ for all pairs $(a,b) \in A \times A$.
An infinite word of the form $h\bparen{f^{\omega}(a)}$, where $f$ is primitive and $h \colon A^{*} \rightarrow B^{*}$ is an arbitrary morphism, 
is called a \emph{primitive morphic} word.

It is clear that all ultimately periodic sequences are morphic.
The following result on morphic words is well-known,
see Theorems~7.6.1 and~7.6.3  and Corollary~7.7.5 in~\cite{AllSha2003}.

\begin{lemma}\label{ke 19-12-2012 12:32}
Let $\bfx \in A^{\N}$ be a morphic word, $w\in A^*$, and $g \colon A^* \rightarrow B^*$  a non-erasing morphism.
Then the words $w \bfx$, $w^{-1} \bfx$, and $g(\bfx)$ are morphic.
\end{lemma}

Let $\bfx \in A^{\N}$ be an infinite word. The set of factors of $\bfx$ is denoted by $F(\bfx)$.
We denote by $\calS_{\bfx}$ the set of all infinite words $\bfy \in A^{\N}$ such that $F(\bfy) \subseteq F(\bfx)$. Thus $\calS_{\bfx}$ is the \emph{shift orbit closure} of $\bfx$.
Endowed with the shift map $T \colon A^{\N} \rightarrow A^{\N}$, $S_{\bfx}$ becomes a symbolic dynamical system, more precisely a subshift,
and we denote this $(S_{\bfx}, T)$.

Now we are ready for a key definition of this paper.
Let $f \colon A^{*} \rightarrow B^{*}$ be a morphism and $\bfx \in A^{\N}$. We will write 
\begin{equation}\label{ti 15-01-2013 09:55}
f \in \calM_{\bfx} 
\end{equation}
if the following condition holds:
for each letter $b\in A$, there exists a finite word $p_{b} \in B^{+}$ such that if $\bfy \in \calS_{\bfx}$ begins with~$b$, then $f(\bfy)$ begins with $p_{b}$, and if $a\in A$ with $a\neq b$, then neither of $p_{a}$ and $p_{b}$ is a prefix of the other. Notice that then $f$ is necessarily non-erasing.

\begin{example}
Let us illustrate the above definition with a morphism appearing in~\cite{Krieger2008}.
Let $f$ be given by $\o \mapsto \o\ka$, $\i \mapsto \o\ka$, and $\ka \mapsto \i$, and let $\bfx$ be the unique fixed point of $f$. 
It is easy to see that if $\o \bfy \in \calS_{\bfx}$, then $\bfy$ must begin with $\ka$; hence $f(\o \bfy)$ begins with $\o\ka \i$.
Similarly, if $\i \bfy \in \calS_{\bfx}$, then $\bfy$ must begin with $\o$; hence $f(\i \bfy)$ begins with $\o\ka \o$. Finally, 
$f(\ka \bfy)$ begins with $\i$ regardless of $\bfy$. Therefore we may let $p_{\o} = \o \ka \i$, $p_{\i} = \o\ka \o$, and $p_{\ka} = \i$, and consequently $f \in \calM_{\bfx}$.
\end{example}

\begin{example}
Let $f$ be the morphism 
$\o \mapsto \o\i\o$, $\i \mapsto \ka\i$, $\ka \mapsto \ka\i\i$,
and let $\bfx = f^{\omega}(\o)$. Now we have $f \not \in \calM_{\bfx}$ because 
$f(\i\o\cdots) = \ka\i\o\cdots $ and $f(\i\ka\cdots) = \ka\i\ka\cdots$, so that if $p_{\i}$ existed, it would have to be a prefix of $\ka\i$,
which is a prefix of $f(\ka)$. Therefore no matter how $p_{\ka}$ is chosen, one of $p_{\i}$ and $p_{\ka}$ is necessarily a prefix of the other.%
\end{example}

Let $\sigma = \sigma_A$ be a total order on an alphabet $A$, that is, a transitive and antisymmetric relation for which either $(a,b)$ or $(b,a)$ is in $\sigma$ for all distinct letters $a,b \in A$. If $(a,b) \in \sigma$, we denote $a <_{\sigma} b$.
The order $\sigma$ extends to a lexicographic order on finite and infinite words over $A$ in the usual way.
Let $a \in A$ be a letter and $\bfx\in A^{\N}$ an infinite word in which $a$ occurs.
Then there exists a unique lexicographically smallest word in $\calS_{\bfx}$ with respect to $\sigma$ that begins with the letter~$a$,
and we will denote it by 
\[
\bfl_{a, \sigma,\bfx}.
\] 
Words of this form are  collectively called the \emph{extremal words} of $\bfx$ or $\calS_{\bfx}$.
We also denote by $\bfs_{a, \sigma, \bfx}$ the infinite word obtained from $\bfl_{a, \sigma, \bfx}$ by erasing the first letter, that is,
\[
\bfl_{a, \sigma,\bfx} = a \, \bfs_{a, \sigma, \bfx}.
\]

For the remainder of this section, let us fix a morphism $f \colon A^{*} \rightarrow A^{*}$. 
A word $u \in A^{*}$ is called \emph{bounded under}~$f$ if there exists a constant $k > 0$ such that $\nabs{f^{n}(u)} < k$ for all $n\geq 0$.
It is clear that every letter occurring in a bounded word is bounded. Let $B_{f} \subset A$ denote the set of bounded letters; the letters in $C_{f}: = A\setminus B_{f}$ are said to be \emph{growing under}~$f$.

The following result is proved in~\cite[Prop.~4.7.62]{CasNic2010}.

\begin{lemma}\label{ti 18-12-2012 15:05}
Suppose that  $\bfx\in A^{\N}$ is a pure morphic word generated by $f$.
There exists a finite subset $Q$ of $C_{f} \times B_{f}^{*} \times B_{f}^{*} \times B_{f}^{*} \times B_{f}^{*} \times B_{f}^{*} \times C_{f}$
such that $F(\bfx)\cap C_{f} B_{f}^{*} C_{f}$ equals the set of all words of the form $c_{1} y_{1}z_{1}^{k} x z_{2}^{k} y_{2} c_{2}$ with 
$(c_{1}, y_{1}, z_{1}, x, z_{2}, y_{2}, c_{2}) \in Q$ and $k\in \N$.
\end{lemma}

\begin{lemma}\label{ke 19-12-2012 11:58}
Suppose that  $\bfx\in A^{\N}$ is a pure morphic word generated by $f$.
If $\bfz \in \calS_{\bfx} \cap B_{f}^{\N}$, then
$\bfz$ is ultimately periodic.
\end{lemma}
\begin{proof}
Suppose that $\bfz \in \calS_{\bfx} \cap B_{f}^{\N}$.
If $\bfx$ has a suffix that is in $B_{f}^{\N}$, then it  is ultimately periodic, which is proved in~\cite[Lemma 4.7.65]{CasNic2010}, and then so is~$\bfz$.
Therefore we may assume that there are infinitely many occurrences of growing letters in~$\bfx$.
Let $u_{n}$ be a sequence of factors of $\bfx$ such that 
$u_{n}$ is a prefix of $u_{n+1}$ for all $n\geq1$ and 
$\bfz = \lim_{\ntoinf} u_{n}$.
Since the first letter of $\bfx$ is necessarily growing and $\bfx$ has infinitely many occurrences of growing letters, it follows that each $u_{n}$ 
is a factor of a word $w_{n}$ such that $w_{n} \in C_{f}B_{f}^{+}C_{f} \cap F(\bfx)$. Since the set $Q$ in Lemma~\ref{ti 18-12-2012 15:05} is finite, there exist letters  $c_{1}, c_{2} \in C_{f}$ and words $y_{1}, y_{2}, z_{1}, z_{2}, x \in B_{f}^{*}$ such that $w_{n_{k}} = c_{1} y_{1} z_{1}^{i_{k}} x z_{2}^{i_{k}} y_{2} c_{2}$ for some subsequence~$n_{k}$.
By chopping off a prefix of length $\nabs{y_{1}}$ and a suffix of length $\nabs{y_{2}}$ from $u_{n_{k}}$ if necessary, we may assume that each sufficiently long $u_{n_{k}}$ is a factor of the biinfinite word $\bfq := {}^{\omega}z_{1}. x z_{2}^{\omega}$, where the the word $x$ occurs in position~$0$.
Now we have two possibilities: If there exists an integer $j \in \Z$ such that infinitely many $u_{n_{k}}$ occurs in $\bfq$ in a position $\geq j$,
then $\lim_{\ktoinf} u_{n_{k}}$ has suffix $z_{2}^{\omega}$. If no such $j$ exists, then $\lim_{\ktoinf} u_{n_{k}}$ has suffix $z_{1}^{\omega}$. 
In the first case $\bfz$ has suffix $z_{2}^{\omega}$ and in the second case it has suffix $z_{1}^{\omega}$.
\end{proof}

Let $M_{f}\subset A$ denote the set of letters  $b$ such that $f^{i}(b)=\epsilon$ for some integer $i\geq 1$, and let $t \geq1$ be the smallest integer such that $f^{t}(b) = \epsilon$ for all $b \in M_{f}$.
Let 
\[
G_{f} = \nset{f^{t}(a)}{ \text{$a \in A$ such that $f(a) = x a y$ for some $x,y \in M_{f}^{*}$}}
\]
Notice that each word $f^{t}(a)$ in $G_{f}$ is a finite fixed point of $f$ because 
\[
f^{t}(a) = f^{t-1}(x) \cdots f(x) x a y f(y) \cdots f^{t-1}(y).
\]
In particular, all words in $G_{f}$ are bounded. The following result is by Head and Lando~\cite{HeaLan1986}, see also~\cite[Theorem 7.3.1]{AllSha2003}.

\begin{lemma}\label{la 12-01-2013 10:20}
Let $\bft \in A^{\N}$ be an infinite word. We have $f(\bft) = \bft$ if and only if at least one of the following two conditions holds:
\begin{enumerate}[(a)]
\item $\bft \in G_{f}^{\omega}$; or
\item $\bft = w 	f^{t-1}(x) \cdots f(x) x a y f(y) f^{2}(y) \cdots$ for some $w\in G_{f}^{*}$ and  $a\in A$ such that $f(a) = x a y$ with $x\in M_{f}^{*}$ and $y\notin M_{f}^{*}$.
\end{enumerate}
\end{lemma}

\begin{lemma}\label{la 12-01-2013 10:41}
Let $f \colon A^{*} \rightarrow A^{*}$ be a morphism. If $\bft \in A^{\N}$ can be written in the form $\bft = w x f(x) f^{2}(x)f^{3}(x) \cdots$, where  $w \in A^{*}$ and $x \notin M_{f}^{*}$, then $\bft$ is morphic.
\end{lemma}
\begin{proof}
Let $b$ be a new letter that does not occur in $A$. Then the infinite word $b x f(x) f^{2}(x)\cdots$ is morphic as it is generated by a morphism
$g \colon (A\cup\{b\})^{*} \rightarrow (A \cup \{b\})^{*}$ for which $g(b) = bx $ and $g(a) = f(a)$ for all $a \in A$.
Thus it follows from Lemma~\ref{ke 19-12-2012 12:32} that $\bft$ is morphic.
\end{proof}

\begin{theorem}\label{su 13-01-2013 16:00}
Let $f \colon A^{*} \rightarrow A^{*}$ be a morphism, and suppose that $\bfx \in A^{\N}$ is a pure morphic word generated by $f$.
If $\bft \in \calS_{\bfx}$ satisfies $f(\bft) = \bft$, then $\bft$ is morphic.
\end{theorem}
\begin{proof}
According to Lemma~\ref{la 12-01-2013 10:20}, either $\bft$ is in $G_{f}^{\omega}$ or 
it is of the form 
\[
\bft = w f^{t-1}(x) \cdots f(x) x a y f(y) f^{2}(y) \cdots.
\] 
In the former case $\bft \in B_{f}^{\N}$, so $\bft$ is ultimately periodic by Lemma~\ref{ke 19-12-2012 11:58},
and thus morphic. In the latter case $\bft$ is morphic by Lemma~\ref{la 12-01-2013 10:41}.
\end{proof}

\section{First main result}

\begin{lemma}\label{ma 15-10-2012 13:24}
Let $\bfx \in A^{\N}$ be an infinite word and $f \colon A^{*} \rightarrow B^{*}$ a morphism.
If $\bfy \in B^{\N}$ is in $\calS_{f(\bfx)}$, then there exist a letter $a \in A$ and an infinite word $\bfz$ such that
$a \bfz \in \calS_{\bfx}$ and $\bfy = u f(\bfz)$, where $u$ is a nonempty suffix of $f(a)$.
\end{lemma}
\begin{proof}
Let $L_{n}$ denote the length-$n$ prefix of $\bfy$; then $L_n$ is a factor of $f(\bfx)$ by the definition of $\calS_{f(\bfx)}$.
Consequently, if $n \geq \max_{a\in A}\nabs{f(a)}$, there exist letters $a_n, b_n \in A$ and a word $v_n \in A^{*}$ such that
$a_n v_n b_{n}$ occurs in $\bfx$ and we have $L_n = s_n f(v_n) p_n$, where $s_n$ is a nonempty suffix of $f(a_n)$ and $p_n$ is a possibly empty prefix of $f(b_{n})$.  
Since there are only finitely many different possibilities for $a_n$ and $s_n$, there exists a letter $a\in A$ and a word $u$ such that $a_{n_{i}} = a$ and $s_{n_{i}} = u$ for infinitely many $n_{i}$. 
The set of words $\{v_{n_i}\}$ being infinite, K\"onig's  Lemma implies that there exists an infinite word $\bfz$ such that 
every prefix of $\bfz$ is a prefix of some $v_{n_{i}}$.
Since each of $a v_{n_{i}}$ is a factor of $\bfx$, we have $a \bfz \in \calS_{\bfx}$.
Furthermore, since each prefix $z$ of $\bfz$ is a prefix of some $v_{n_{i}}$, the word $uf(z)$ is a prefix of $\bfy$, and consequently $\bfy = uf(\bfz)$.
\end{proof}

\begin{lemma}\label{pe 02-11-2012 14:01}
Let  $f\colon A^{*} \rightarrow B^{*}$ be a morphism and $\bfx\in A^{\N}$ such that $f \in \calM_{\bfx}$.
Let $b \in B$ be a letter  that occurs in $f(\bfx)$ and let $\rho$ be a total order on $B$.
Then there exist a total order $\sigma$ on $A$, a letter $a \in A$,
and a possibly empty proper suffix $v$ of $f(a)$ such that
\begin{equation}\label{ti 27-11-2012 15:55}
\bfs_{b, \rho, f(\bfx)} = v f(\bfs_{a, \sigma, \bfx}).
\end{equation}
\end{lemma}
\begin{proof}
By Lemma~\ref{ma 15-10-2012 13:24}, we can write  $\bfl_{b, \rho,f(\bfx)} = b \bfs_{b, \rho, f(\bfx)} =  u f(\bfz)$, where $u$ is a nonempty suffix of $f(a)$ for some $a\in A$ and 
$a\bfz \in \calS_{\bfx}$.
Since $f \in \calM_{\bfx}$,  there exist words $p_x \in B^+$ for every $x \in A$ such that  $p_x \neq p_y$ whenever $x \neq y$. 
Thus we can define a total order $\sigma$ on $A$ such that, for all letters $x,y \in A$, we have $x <_{\sigma} y$ if and only if $p_{x} <_{\rho} p_{y}$. 

We claim that $\bfz = \bfs_{a, \sigma, \bfx}$.  If this is not the case, then  $\bfz >_{\sigma} \bfs_{a, \sigma, \bfx}$ because both $a\bfz$ and $a\bfs_{a, \sigma, \bfx} = \bfl_{a, \sigma, \bfx}$
are in $\calS_{\bfx}$ and $ \bfl_{a, \sigma, \bfx}$ is the smallest word in $\calS_\bfx$ starting with the letter~$a$.
Therefore  $\bfz = w y \bft$ and $\bfs_{a, \sigma, \bfx} = w x \bft'$ with   $x,y\in A$ satisfying $y  >_{\sigma} x$. 
Since $f(x \bft)$  begins with $p_{x}$ and $f(y \bft')$  begins with $p_{y}$ and neither of $p_x$ and $p_y$ is a prefix of the other,
we have $ f(y \bft) >_{\rho} f(x\bft')$ by the definition of $\sigma$, and this gives
\[
 \bfl_{b, \rho,f(\bfx)}  =  u f(\bfz) = u f(w) f(y \bft) >_{\rho} u f(w) f(x \bft') = u f(\bfs_{a, \sigma, \bfx}).
\]
But this contradicts the definition of $\bfl_{b, \rho,f(\bfx)}$ because $uf(\bfs_{a, \sigma, \bfx})$ starts with the letter $b$ and is in $\calS_{f(\bfx)}$.
Therefore we have shown that $\bfz = \bfs_{a, \sigma, \bfx}$, and so~\eqref{ti 27-11-2012 15:55} holds with $v = b^{-1} u$.
\end{proof}

\begin{lemma} \label{su 04-11-2012 12:51}
Let  $f\colon A^{*} \rightarrow A^{*}$ be a morphism and $\bfx\in A^{\N}$ such that $f \in \calM_{\bfx}$ and $f(\bfx) = \bfx$.
Then for any total order $\rho$ on $A$ and any letter $b \in A$ occurring in~$\bfx$, there
exist a total order $\sigma$ on $A$, a letter $a \in A$,  words $u, v \in A^{*}$, and integers $k,m\geq 1$ such that 
\begin{equation} \label{to 29-11-2012 10:59}
\bfs_{b, \rho, \bfx} = u f^{k}(\bfs_{a, \sigma, \bfx}) \qqtext{and} \bfs_{a, \sigma, \bfx} = v f^{m}(\bfs_{a, \sigma, \bfx}).
\end{equation}
\end{lemma}
\begin{proof}
Since $f(\bfx) = \bfx$, Lemma~\ref{pe 02-11-2012 14:01} implies that 
$\bfs_{b, \rho, \bfx} = v_0 f(\bfs_{a_1, \sigma_1, \bfx})$ 
for some total order $\sigma_1$ on $A$, a letter $a_1 \in A$, and a possibly empty suffix $v_0$ of $f(a_1)$.
By applying Lemma~\ref{pe 02-11-2012 14:01} next on $\bfs_{a_1, \sigma_1, \bfx}$ and further, we get a sequence of identities
\[
\bfs_{a_{k}, \sigma_k, \bfx} = v_{k} f(\bfs_{a_{k+1}, \sigma_{k+1}, \bfx})     \qquad  (k\geq0),
\]
where we  denote $a_{0} = b$ and $\sigma_{0} = \rho$.
Therefore,
\[
\bfs_{a_{k}, \sigma_{k}, \bfx}
= v_{k} f(v_{k+1}) \cdots f^{m - 1 }(v_{k  + m - 1}) f^{m}(\bfs_{a_{k+m}, \sigma_{k+m}, \bfx}),
\]
for all integers $k\geq 0$ and $m\geq 1$. 
Since there are only finitely many different letters and total orders on $A$,
there is a choice  of $k$ and~$m$ such that
$a_{k} = a_{k+m}$ and $\sigma_{k} = \sigma_{k+m}$.
Thus by denoting $a = a_k$, $\sigma = \sigma_k$,  
\[
u = v_{0} f(v_{1}) \cdots f^{k - 1 }(v_{k  - 1}), \qtext{and} v = v_{k} f(v_{k+1}) \cdots f^{m -1}(v_{k + m - 1}),
\]
we have the identities in~\eqref{to 29-11-2012 10:59}.
\end{proof}

\begin{lemma}\label{ke 07-11-2012 17:01}
Let  $f\colon A^{*} \rightarrow A^{*}$ be a morphism and $\bfx\in A^{\N}$ such that $f \in \calM_{\bfx}$ and $f(\bfx) = \bfx$.
Then for any total order $\rho$ on $A$ and any letter  $b \in A$ occurring in $\bfx$,
there exist a finite word $w\in A^{+}$, an infinite word $\bft \in \calS_{\bfx}$, and an integer $m\geq 1$ such that
\begin{align}
\bfl_{b, \rho, \bfx} &= w \bft \label{pe 30-11-2012 13:06}\\
\intertext{and either}
 \bft &= f^m(\bft) \label{pe 30-11-2012 13:07} \\
\intertext{or}
\bft &= \lim_{\ntoinf} x f^m(x) f^{2m}(x) \cdots f^{nm}(x) \label{pe 30-11-2012 13:08}
\end{align}
for some finite word $x \in A^+$.
\end{lemma}
\begin{proof}
According to Lemma~\ref{su 04-11-2012 12:51}, 
there exist a total order $\sigma$ on $A$, a letter $a \in A$,  words $u, v \in A^{*}$, and integers $k,m\geq 1$ such that 
\[
\bfs_{b, \rho, \bfx} = u f^{k}(\bfs_{a, \sigma, \bfx}) \qqtext{and} \bfs_{a, \sigma, \bfx} = v f^{m}(\bfs_{a, \sigma, \bfx}).
\]
Denote $w = bu$ and $\bft = f^k(\bfs_{a, \sigma, \bfx})$. Then $\bft \in \calS_{\bfx}$ and Eq.~\eqref{pe 30-11-2012 13:06} holds.
By denoting $x = f^k(v)$, we get $\bft = x f^m(\bft)$.
If $x =\epsilon$, then we have $\bft = f^m(\bft)$, and Eq.~\eqref{pe 30-11-2012 13:07} holds.
If $x \neq \epsilon$, then
\[
\bft = xf^m(\bft) = x f^m(x) f^{2m}(\bft) = \cdots = x f^m(x) f^{2m}(x) \cdots f^{nm}(x) f^{(n+1)m}(\bft),
\]
for all integers $n\geq 0$.
The morphism $f$ is non-erasing because $f\in \calM_{\bfx}$, and therefore the words $x f^m(x) f^{2m}(x) \cdots f^{nm}(x)$ get longer and longer as $n$ grows.
Thus Eq.~\eqref{pe 30-11-2012 13:08} holds.
\end{proof}

Here is the first main result of this paper.

\begin{theorem}\label{su 04-11-2012 12:52}
Let  $f\colon A^{*} \rightarrow A^{*}$ be a morphism. If $\bfx\in A^{\N}$ 
is a pure morphic word generated by~$f$ and $f \in \calM_{\bfx}$, 
then all extremal words in $\calS_{\bfx}$ are morphic.
\end{theorem}
\begin{proof}
Let $\rho$ be a total order on $A$ and $b \in A$ a letter occurring in~$\bfx$. We will show that $\bfl_{b, \rho, \bfx}$ is morphic. 
Lemma~\ref{ke 07-11-2012 17:01} says that there exist a finite word $w\in A^{+}$, an infinite word $\bft \in \calS_{\bfx}$, and an integer $m\geq 1$ such that
$\bfl_{b, \rho, \bfx} = w \bft$ and either 
$ \bft = f^m(\bft)$ or $\bft = \lim_{\ntoinf} x f^m(x) f^{2m}(x) \cdots f^{nm}(x)$ for some finite word $x\in A^+$.
Since $f^{m}$ generates $\bfx$, the claim that $\bft$ is morphic follows in the former case from Theorem~\ref{su 13-01-2013 16:00} 
and in the latter case from Lemma~\ref{la 12-01-2013 10:41}.

\end{proof}

\begin{theorem}\label{ti 04-12-2012 15:36}
Let  $f\colon A^{*} \rightarrow A^{*}$ and $g \colon  A^{*} \rightarrow B^{*}$  be  morphisms
and $\bfx\in A^{\N}$ such that $f, g \in \calM_{\bfx}$.
If $\bfx$ is a pure morphic word generated by $f$, then all extremal words in $\calS_{g(\bfx)}$ are morphic.
\end{theorem}
\begin{proof}
Let $\rho$ be a total order $\rho$ on $B$ and $b \in B$. 
According to Lemma~\ref{pe 02-11-2012 14:01}, there exists a total order $\sigma$ on $A$, a letter $a\in A$, and a word $v\in B^{*}$ such that
\[
\bfs_{b, \rho, g(\bfx)} = v g(\bfs_{a, \sigma, \bfx})
\]
Thus it follows from Theorem~\ref{su 04-11-2012 12:52} and Lemma~\ref{ke 19-12-2012 12:32} that 
$\bfl_{b, \rho, g(\bfx)}$ is morphic.
\end{proof}

\section{Extremal words of binary pure morphic words}

In this section we show that the extremal words of binary pure morphic words are morphic.

\begin{lemma}\label{su 04-11-2012 12:54}
Let $f \colon \bina^{*} \rightarrow \bina^{*}$ be a morphism such that $f(\o\i) \neq f(\i\o)$.  
Then $f \in \calM_{\bfx}$ for every $\bfx \in \bina^{\N}$.
\end{lemma}
\begin{proof}
Let us denote $u = f(\o)$ and $v = f(\i)$. We have two possibilities:

\emph{Case 1.} The word $u$ is not a prefix of $v^{\omega}$. Then there exists an integer $n\geq 0$ such that 
$u = v^{n} pas$ and $v = pbt$, where $p,s,t \in \bina^{*}$ and $a,b \in \bina$ with  $a \neq b$.
Now it is easy to see that, for every $\bfy \in \bina^{\N}$, the word $f(\i\bfy)$ begins with $v^{n}pb$ and $f(\o \bfy)$ begins with $v^{n}pa$.
Therefore $f \in \calM_{\bfx}$ because we may choose $p_{\i} = v^{n}pb$ and $p_{\o} = v^{n}pa$.

\emph{Case 2.} The word $u$ is a prefix of $v^{\omega}$. Then $v = xy$ and $u = v^{n}x$ for some integer $n\geq0$ and words $x,y$.
Now it is easy to see that, for every $\bfy \in \bina^{\N}$, the word $f(\o \bfy)$ begins with $(xy)^{n}xxy$ and $f(\i \bfy)$ begins with $(xy)^{n}xyx$.
Since $f(\o\i) \neq f(\i\o)$, it follows that $xy \neq yx$. Denote $xy = pas$ and $yx = pbt$ with $a,b$ distinct letters.
Then $f \in \calM_{\bfx}$ because we may let $p_{\o} = (xy)^{n}xpa$ and $p_{\i} = (xy)^{n}xpb$.
\end{proof}

\begin{theorem}\label{ma 14-01-2013 10:22}
If $\bfx \in \bina^{\N}$ is a binary pure morphic word,
then all extremal words of $\bfx$ are morphic.
\end{theorem}
\begin{proof}
Let $f$ be a binary morphism that generates $\bfx$. If $f(\o\i) = f(\i\o)$, then $\bfx$ is purely periodic, and the claim holds.
If $f(\o\i) \neq f(\i\o)$, then $f \in \calM_{\bfx}$ by Lemma~\ref{su 04-11-2012 12:54}, so that $\bfx$ is morphic by Theorem~\ref{su 04-11-2012 12:52}.
\end{proof}

There are exactly two total orders on the binary alphabet $\{\o,\i\}$;
let~$\rho$ denote the natural order $\o <_{\rho} \i$ and $\overline{\rho}$ the other order $\i <_{\overline{\rho}} \o$.
The following lemma simplifies the search for the extremal words of a binary pure morphic word, and we will
use it later.

\begin{lemma}\label{la 08-12-2012 22:28}
If $\bfx \in \{\o, \i\}^{\N}$ is a recurrent word in which both $\o$ and $\i$ occur, then
\begin{align}
\bfl_{\i, \rho, \bfx} &= \i \bfl_{\o, \rho, \bfx}   &  \bfl_{\o, \overline{\rho}, \bfx} & =  \o \bfl_{\i, \overline{\rho}, \bfx}.\label{ke 19-12-2012 16:11}\\
\intertext{Therefore also}
\bfs_{\i, \rho, \bfx} &= \o \bfs_{\o, \rho, \bfx}   &  \bfs_{\o, \overline{\rho}, \bfx} &= \i \bfs_{\i, \overline{\rho}, \bfx}.\label{ti 15-01-2013 13:49}
\end{align}
\end{lemma}
\begin{proof}
Consider the first equation in~\eqref{ke 19-12-2012 16:11}. On the one hand, $\i \bfl_{\o, \rho, \bfx}$ is in $\calS_{\bfx}$ because  the recurrence of $\bfx$ implies that  
$a \bfl_{\o, \rho, \bfx}$ is in $\calS_{\bfx}$ for some $a\in \{\o, \i\}$ and if $a$ equaled $\o$, then the inequality $\o \bfl_{\o, \rho,\bfx} < \bfl_{\o, \rho,\bfx}$ would contradict the definition of $\bfl_{\o, \rho,\bfx}$. On the other hand,  $\i \bfl_{\o, \rho, \bfx}$ must equal $\bfl_{\i, \rho, \bfx}$ because otherwise
$\bfl_{\i, \rho, \bfx} < \i \bfl_{\o, \rho, \bfx}$, which implies $\i^{-1}\bfl_{\i, \rho, \bfx} < \bfl_{\o, \rho, \bfx}$, and this  contradicts the definition of $\bfl_{\o, \rho, \bfx}$. The second equation in~\eqref{ke 19-12-2012 16:11} is proved similarly. The identities~\eqref{ti 15-01-2013 13:49} follow immediately from~\eqref{ke 19-12-2012 16:11}.
\end{proof}

\section{Second main result}

In this section we assume that $\bfx = x_1 x_2 x_3 \cdots$, where $x_i \in A$, is a primitive morphic word and show
that the extremal words of $\bfx$ are primitive morphic. The techniques used in this section are completely different from those used in the previous sections, and they rely on the notion of return words.

Let $u$ be a factor of~$\bfx$. Let $i$ and $j$ be positive integers such that $i< j$ and the word $u$ 
is a prefix of $x_i x_{i+1} x_{i+2} \cdots$ and $x_j x_{j+1} x_{j+1} \cdots$ but not a prefix of $x_k x_{k+1} x_{k+2}\cdots$  whenever $i < k < j$.
Then the words 
$x_i x_{i+1} \cdots x_{j-1}$ 
and
$x_i x_{i+1} \cdots x_{j-1}u$
are called, respectively, a \emph{first return} and a \emph{first complete return} to~$u$.
The former is also simply referred to as a \emph{return word} of~$u$.

The set of all return words of $u$ in $\bfx$  is denoted by $\calR_u$.
Since $\bfx$ is primitive morphic,  the number of return words of $u$ is finite; see~\cite{Durand1998}.
Write $A_u = \nsset{1, 2, \ldots, \nabs{\calR_u}}$

Fix a bijection $\sigma_u \colon A_u  \rightarrow \calR_u$.  If $\bfy$ is in the shift orbit closure $\calS_{\bfx}$  of $\bfx$ and begins with $u$, then $\bfy$ can be uniquely factorized over the set $\calR_u$.
Thus there exists a unique sequence $\calD_u(\bfy) = a_1 a_2 a_3 \cdots$ with $a_i \in A_u$ such that
$\bfy = \sigma_u(a_1) \sigma_u(a_2) \sigma_u(a_3) \cdots$.
We call $\calD_u(\bfy)$ a \emph{derived word} of $\bfy$.

Let us define
\[
X_u = \bset{ \calD_u(\bfy) }{ \bfy \in \calS_{\bfx}, \quad \text{$\bfy$ begins with $u$}}
\]
The symbolic dynamical system  $(X_u, T)$, where $T$ is the shift map of infinite words,  is a so-called induced system 
of the subshift $(\calS_{\bfx}, T)$.  A result by  Holton and Zamboni~\cite[Thm.~8.2]{HolZam1999} says that the number 
of different sets $X_u$ is finite.

Another result we will need is the following characterization of primitive morphic words by Durand~\cite{Durand1998}:
An infinite word is primitive morphic if and only if the number of its derived  words is finite.

Now we are ready to prove the second main result of this paper.

\begin{theorem}\label{2013-07-13 19:29}
Let $\bfx \in A^{\N}$ be a primitive morphic word.  If $\bfy$ is an extremal word in the shift orbit closure of $\bfx$, then $\bfy$ is primitive morphic.
\end{theorem}
\begin{proof}
Suppose that $\bfy$ is extremal with respect to a total order $\leq$ on $A$.
Let $u$ be a prefix of $\bfy$.
There is a natural linear ordering $\leq_u$ on the set $A_u$ given by: for $m,n\in A_u$, write
$m \leq_u n$ if and only if $\sigma_u(m) u \leq \sigma_u(n) u$. Note that as $\sigma_u(m)u$ and $\sigma_u(n) u$ are distinct complete returns to $u$, they are never prefixes of one another and hence are comparable. 
It is readily checked that the derived word 
$\calD_u(\bfy)$ is an extremal word in $X_u$ with respect to $\leq_u$ for each prefix $u$ of $\bfy$.
Since  the number of extremal words in $X_u$ is finite, and there are only finitely many sets $X_u$, this means that there are only finitely
many derived words $\calD_u(\bfy)$ when $u$ is a prefix of $\bfy$. Thus $\bfy$ is primitive morphic.
\end{proof}

\section{Extremal words of the period-doubling word}\label{perioddoubling}

Let $f$ denote the morphism $\o \mapsto \o \i$,  $\i \mapsto \o\o$ and let $\bfd = f^{\omega}(\o)$ denote the 
\emph{period-doubling word}~\cite{Damanik2000,AllSha2003,Makarov2010}. According to Lemma~\ref{su 04-11-2012 12:54}, we have $f \in \calM_{\bfd}$.

Let $\rho$ denote the natural order  $\o <_{\rho} \i$ and $\overline{\rho}$ the reversed order $\i <_{\overline{\rho}} \o$.
Using the observation that neither $\o\o\o\o$ nor $\i\i$ occur in $\bfd$ and Lemma~\ref{la 08-12-2012 22:28},
the reader has no trouble verifying that  the following words start as shown.
\begin{align}
\bfs_{\o, \rho, \bfd} &=  \word{00100\cdots} & \bfs_{\i, \overline{\rho}, \bfd} &=  \word{010100\cdots} \label{ti 2013-01-15 15:08} \\ 
\bfs_{\i, \rho, \bfd} &= \word{0001\cdots}  & \bfs_{\o, \overline{\rho}, \bfd}  &=  \word{1010100\cdots}.\label{pe 2013-01-04 15:51}
\end{align}

Lemma~\ref{pe 02-11-2012 14:01} implies that $\bfs_{\o, \rho, \bfd} = v f(\bfs_{a, \sigma, \bfd})$ for some $a\in \bina$, proper suffix $v$ of $f(a)$, and $\sigma \in \{\rho, \overline{\rho}\}$. The only possible such factorization has to be of the form $\bfs_{\o, \rho, \bfd} =  \o f(\o\i \cdots)$,
so from \eqref{ti 2013-01-15 15:08} and \eqref{pe 2013-01-04 15:51} we see that $\bfs_{a, \sigma, \bfd} = \bfs_{\i, \overline{\rho},\bfd}$. Thus
\[
\bfs_{\o, \rho, \bfd} = \o f(\bfs_{\i, \overline{\rho},\bfd}). 
\]
We can deduce similarly that 
\begin{equation}\label{ke 16-01-2013 09:46}
\bfs_{\i, \overline{\rho}, \bfd} =  f(\o\o\i \cdots) =  f(\bfs_{\o, \rho, \bfd}). 
\end{equation}
Therefore $\bfs_{\o, \rho, \bfd} = \o f^{2}(\bfs_{\o, \rho, \bfd})$, which implies
\begin{equation}\label{pe 2013-01-05 15:05}
f^{2}(\bfl_{\o, \rho, \bfd}) = \o\i \bfl_{\o, \rho, \bfd}.
\end{equation}
We claim that $\bfl_{\o, \rho, \bfd}$ is the fixed point of the morphism $g \colon \o \mapsto \o\o\o\i$ and $\i \mapsto \o\i\o\i$.
Let us denote the unique fixed point of $g$ by $\bfz$, that is $\bfz = g^{\omega}(\o)$.
An easy induction proof shows that $\o\i g(w) = f^{2}(w) \o\i$ for all $w \in \bina^{*}$. Therefore
\[
\o\i \bfz = \o\i g(\bfz) = f^{2}(\bfz).
\]
Thus by \eqref{pe 2013-01-05 15:05}, both $\bfz$ and $\bfl_{\o, \rho, \bfd}$ satisfy the same relation $\o\i\bfx= f^{2}(\bfx)$,
which is easily seen to admit a unique solution; thus $\bfz = \bfl_{\o, \rho, \bfd}$.
Hence, using \eqref{ke 16-01-2013 09:46} and Lemma~\ref{la 08-12-2012 22:28}, the following result is obtained.
\begin{theorem} \label{14-01-2013 15:05}
Let $\bfd$ denote the period-doubling word and let $\bfz$ denote the unique fixed point of the morphism $\o \mapsto \o\o\o\i$, $\i \mapsto \o\i\o\i$.
Then we have
\begin{align*}
\bfl_{\o, \rho, \bfd} &= \bfz &  \bfl_{\i, \rho, \bfd} &= \i \bfz \\
\bfl_{\i, \overline{\rho}, \bfd} &= \o^{-1}f(\bfz)  & \bfl_{\o, \overline{\rho}, \bfd} &= f(\bfz).
\end{align*}
\end{theorem}

\section{Extremal words of the Chacon word}

The Chacon word \cite{Ferenczi1995,Fogg2002} is the fixed point $\bfc = f^{\omega}(\o)$, where $f$ is the morphism
$\o \mapsto \o\o\i\o$, $\i \mapsto \i$. Lemma~\ref{su 04-11-2012 12:54} guarantees that $f \in \calM_{\bfx}$.
Let $\rho$ denote the natural order  $\o <_{\rho} \i$ and $\overline{\rho}$ the reversed order $\i <_{\overline{\rho}} \o$ as  before.

As in Section~\ref{perioddoubling}, we use the observation that neither $\o\o\o\o$ nor $\i\i$ occur in $\bfc$ and Lemma~\ref{la 08-12-2012 22:28},
to deduce that the following words start as shown.

\begin{align*}
\bfs_{\o, \rho, \bfc} &= \o\o\i\o\o\o\i\o\i \cdots  &   \bfs_{\i, \overline{\rho}, \bfc} &= \o\i\o\o\i\o \cdots \\
\bfs_{\i, \rho, \bfc} &=  \o \o\o\i\o\o\o\i\o\i \cdots &   \bfs_{\o, \overline{\rho}, \bfc} &=  \i \o\i\o\o\i\o \cdots.
\end{align*}
Applying Lemma~\ref{pe 02-11-2012 14:01} as in the previous section, we find
\[
\bfs_{\o, \rho, \bfc} = f(\o\o\i \cdots) = f(\bfs_{\o, \rho, \bfc}).
\]
Since $\bfs_{\o, \rho, \bfc}$ begins with $\o$, we thus have $\bfs_{\o, \rho, \bfc} = \bfc$ and $\bfl_{\o, \rho, \bfc} = \o\bfc$.

Similarly, recalling that $\bfs_{\o, \overline{\rho}, \bfc} = \i \bfs_{\i, \overline{\rho}, \bfc}$ by Lemma~\ref{la 08-12-2012 22:28},
we deduce using Lemma~\ref{pe 02-11-2012 14:01} that
\[
\bfs_{\i, \overline{\rho}, \bfc} = \o f(\i\o \cdots) = \o f(\bfs_{\o, \overline{\rho}, \bfc}) = \o\i f(\bfs_{\i, \overline{\rho}, \bfc}).
\]
Therefore $\bfl_{\i, \overline{\rho}, \bfc}$ can be expressed as $\bfl_{\i, \overline{\rho}, \bfc} = \tau g^{\omega}(b)$, where $b$ is a new symbol,
$g$ is a morphism for which $g(b) = b \o\i$ and $g(a) = f(a)$ for $a\in \bina$, and $\tau(b) = \i$ and $\tau(a) = a$ for $a\in \bina$.
Thus a final application of Lemma~\ref{la 08-12-2012 22:28} allows us to wrap up the results of this section as follows.
\begin{theorem}\label{14-01-2013 15:07}
Let $\bfc$ denote the Chacon word. Then we have
\begin{align*}
\bfl_{\o, \rho, \bfc} &= \o \bfc  &   \bfl_{\i, \overline{\rho}, \bfc} &=  \tau g^{\omega}(b) \\
\bfl_{\i, \rho, \bfc} &=  \i \o \bfc  &   \bfl_{\o, \overline{\rho}, \bfc} &=  \o \tau g^{\omega}(b),
\end{align*} 
where $g$ and $\tau$ are the morphisms given above.
\end{theorem}

\section{The least word in the shift orbit closure of the Rudin-Shapiro word}

In this section, we give a new proof for the form of the lexicographically smallest word in the
shift orbit closure  of the Rudin-Shapiro word. This result was first derived in~\cite{Currie2011}.
Considerations in this section are more involved than the ones in the previous sections because
a coding is needed in the definition of the Rudin-Shapiro word.
In what follows, we denote the natural order on letters $\o, \i, \ka,\ko$ by $\rho$. Thus we have
$\o <_{\rho} \i <_{\rho} \ka <_{\rho} \ko$.

Let $f$ and $g$ be the morphisms
\[
f \colon 
\begin{cases}
\o \mapsto \o\i\\
\i \mapsto \o\ka\\
\ka \mapsto \ko\i\\
\ko \mapsto \ko\ka 
\end{cases}
\qqtext{and}
g \colon
\begin{cases}
\o \mapsto \o\\
\i \mapsto \o\\
\ka \mapsto \i\\
\ko \mapsto \i 
\end{cases}
\]
Denote
\begin{align*}
\bfu = f^{\omega}(\o) &= \word{0102013101023202010201313231013101020131\cdots} \\
\intertext{and}
\bfw = g(\bfu) &= \word{0001001000011101000100101110001000010010\cdots}.
\end{align*}
Then $\bfw$ is the Rudin-Shapiro word, and our goal is to  prove the identity 
$\bfl_{\o, \rho, \bfw} = \o \bfw$.
To that end, we need the next two lemmas. Let us denote $\Sigma_{4} = \{ \o, \i, \ka, \ko\}$.

\begin{lemma}\label{la 15-12-2012 15:14}
Let $\sigma$ and $\sigma'$ be two total orders on  $\Sigma_{4}$. If $\sigma$ and $\sigma'$ order the pairs $(\o,\ko)$ and $(\i, \ka)$ in the same way, 
i.e., $\o <_{\sigma} \ko$ if and only if $\o <_{\sigma'} \ko$ and $\i <_{\sigma} \ka$ iff $\i <_{\sigma'} \ka$,
then $\bfl_{d, \sigma, \bfu} = \bfl_{d, \sigma', \bfu}$ for all $d \in \Sigma_{4}$.
\end{lemma}
\begin{proof}
Suppose that $\bfl_{d, \sigma, \bfu} = u a \bft$ and $\bfl_{d, \sigma', \bfu} = u b \bft'$ with distinct letters $a,b \in \Sigma_{4}$.
Since $\sigma$ and $\sigma'$ agree on $(\o,\ko)$ and $(\i, \ka)$, it follows that either $a \in \{ \o, \ko\}$ and $b \in \{\i, \ka\}$, or vice versa.
Furthermore, if $c$ denotes the last letter of $u$, then both $ca$ and $cb$ occur in $\bfu$. This contradicts the fact that none of 
the words $\o\o$, $\o\ko$, $\i\i$, $\i\ka$, $\ka\i$, $\ka\ka$, $\ko\o$, $\ko\ko$
occur in $\bfu$.
\end{proof}

The next lemma is interesting in its own right. It was also proved in~\cite{Currie2011}.

\begin{lemma}\label{ma 17-12-2012 15:25}
 We have $\bfl_{\o, \rho, \bfu} = \bfu$.
\end{lemma}
\begin{proof}
Since clearly $f \in \calM_{\bfu}$, Lemma~\ref{pe 02-11-2012 14:01} implies that there exist
a letter $a \in \Sigma_{4}$, a proper suffix $v$ of $f(a)$, and a total order $\sigma$ on $\Sigma_{4}$
such that $\bfs_{\o, \rho, \bfu} = v f(\bfs_{a, \sigma, \bfu})$.
An easy case analysis based on the observation that $\o\o$ does not occur in $\bfu$ yields
$\bfs_{\o, \rho, \bfu} = \word{10201}\cdots = \i f(\i\o \cdots)$,
and hence
\[
v = \i \qqtext{and} \bfs_{a, \sigma, \bfu} = \i\o \cdots.
\]
Since $v = \i$ is a suffix of $f(a)$, we have $a=\o$ or $a =\ka$. Furthermore since
$\bfl_{a, \sigma, \bfu}$ starts with $a\i$, we must have 
$a = \o$ because $\ka \i$ does not occur in $\bfu$. Thus $\bfs_{\o, \rho, \bfu} = \i f(\bfs_{\o, \sigma, \bfu})$.

Next we claim  $\bfs_{\o, \sigma, \bfu} = \bfs_{\o, \rho, \bfu}$. We prove this by showing that $\o<_{\sigma} \ko$ and $\i <_{\sigma} \ka$;
then the claim follows from Lemma~\ref{la 15-12-2012 15:14}. If, contrary to what we want to show,  we have $\ka <_{\sigma} \i$, then $\bfl_{\o, \sigma, \bfu}$ would begin with $\o \ka$, contradicting the fact that $\bfs_{\o, \sigma, \bfu}$ begins with $\i$.
Consequently we have $\i <_{\sigma} \ka$. Furthermore if $\ko <_{\sigma} \o$, then $\bfl_{\o, \sigma, \bfu}$ would begin with $\o\i\ko$,
contradicting the fact that $\bfs_{\o, \sigma, \bfu}$ begins with $\i\o$.
Therefore $\bfs_{\o, \sigma, \bfu} = \bfs_{\o, \rho, \bfu}$.

Now the identity  $\bfs_{\o, \rho, \bfu} = \i f(\bfs_{\o, \rho, \bfu})$ implies $\bfl_{\o, \rho, \bfu} =  f(\bfl_{\o, \rho, \bfu})$,
so that $\bfl_{\o, \rho, \bfu}$ is the unique iterative fixed point of $f$ that starts with \o, that is $\bfl_{\o, \rho, \bfu} = \bfu$. 
\end{proof}

Finally, we are ready to prove the main result of this subsection.

\begin{theorem}
Let $\bfw$ denote the Rudin--Shapiro word. Then $\bfl_{\o, \rho, \bfw} = \o \bfw$. 
\end{theorem}
\begin{proof}
Let $h = g \circ f$ be the composition of $g$ and $f$. Then
\[
h \colon \quad \o \mapsto \o\o \qquad \i \mapsto \o\i \qquad \ka \mapsto \i\o \qquad \ko \mapsto \i\i.
\]

According to Lemma~\ref{ma 15-10-2012 13:24}, 
there exist a letter $a \in \Sigma_{4}$ and an infinite word $\bfz \in \Sigma_{4}^{\N}$ such that
$a\bfz \in \calS_{\bfu}$ and $\bfl_{\o, \rho, \bfw} = u h(\bfz)$, where $u$ is a nonempty suffix of $h(a)$.
Since $\bfl_{\o, \rho, \bfw}$ clearly starts with $\o\o\o\o$ and $\o\o$ does not occur in $\bfu$, 
it follows that $u = \o$, $\bfz = \o\i\cdots$, and $a = \ka$. 

On the other hand, it is easy to see that $\ka \bfu \in \calS_{\bfu}$. Since 
$\bfu = \bfl_{\o, \rho, \bfu}$ by Lemma~\ref{ma 17-12-2012 15:25}, we have $\bfu \leq_{\rho} \bfz$, and so $\ka \bfu \leq_{\rho} \ka \bfz$. 
Furthermore, since $h$ preserves~$\rho$,
that is to say if $x,y \in \{\o,\i,\ka, \ko \}^{*}$ with $x <_{\rho} y$, then $h(x) <_{\rho} h(y)$,
we have $h(\ka \bfu) \leq_{\rho} h(\ka \bfz)$, which gives
$\o h(\bfu) \leq_{\rho} \o h(\bfz) = \bfl_{\o, \rho, \bfw}$.
Hence we must have $\o h(\bfu) = \bfl_{\o, \rho, \bfw}$, and so
\[
\bfl_{\o, \rho, \bfw} = \o h(\bfu) = \o g(\bfu) = \o \bfw.
\]
\end{proof}

\section{Conclusion}\label{conclusion}

We have shown that if $\bfx = g\nparen{f^{\omega}(a)}$  such that either $f,g \in \calM_{f^{\omega}(a)}$ or $f$
is primitive, then all extremal words of $\bfx$ are morphic.
We also know from a previous work of Allouche et al.~\cite{AllRamSha2009} that the extremal words of automatic words are  automatic. 
However, there are morphic words that do not fall into any of these classes,
so it remains an open problem whether all extremal words of all morphic words are morphic.

\end{document}